\documentclass[a4paper,12pt]{article}

\usepackage[english]{babel}
\usepackage{tikz}
\usetikzlibrary{matrix,arrows,decorations.markings}
\usepackage{amsmath,amsfonts,amssymb,amsthm,url,textcomp}
\usepackage{csquotes}
\usepackage{a4wide}
\usepackage[numbers]{natbib}
\usepackage{fancyhdr}
\usepackage{anyfontsize}
\usepackage{hyperref}
\usepackage{hhline}
\usepackage{leftidx}
\usepackage[ruled,vlined]{algorithm2e}
\usepackage{enumitem}

\allowdisplaybreaks

\numberwithin{theorem}{section}

\newtheorem{theoremm}{Theorem}\numberwithin{theoremm}{subsection}

\newtheorem{lemmma}[theoremm]{Lemma}

\numberwithin{theoremmm}{subsubsection}

\theoremstyle{remark}

\newcommand{\Aut}{\operatorname{Aut}}

\newcommand{\lcm}{\operatorname{lcm}}

\newcommand{\ord}{\operatorname{ord}}

\newcommand{\A}{\operatorname{A}}

\newcommand{\id}{\operatorname{id}}

\newcommand{\e}{\mathrm{e}}

\newcommand{\GL}{\operatorname{GL}}

\newcommand{\G}{\mathcal{G}}

\newcommand{\Mod}[1]{\ (\textup{mod}\ #1)}

\newcommand{\IN}{\mathbb{N}}

\newcommand{\IF}{\mathbb{F}}

\newcommand{\IZ}{\mathbb{Z}}

\newcommand{\Coll}{\operatorname{Coll}}

\newcommand{\rord}{\operatorname{rord}}
\newcommand{\Acal}{\mathcal{A}}
\newcommand{\Exp}{\operatorname{Exp}}

\begin{document}

\title{Computation of orders and cycle lengths of automorphisms of finite solvable groups}

\author{Alexander Bors\thanks{Johann Radon Institute for Computational and Applied Mathematics (RICAM), Altenberger Stra{\ss}e 69, 4040 Linz, Austria. \newline E-mail: \href{mailto:alexander.bors@ricam.oeaw.ac.at}{alexander.bors@ricam.oeaw.ac.at} \newline The author is supported by the Austrian Science Fund (FWF), project J4072-N32 \enquote{Affine maps on finite groups}. \newline 2010 \emph{Mathematics Subject Classification}: Primary: 20D10, 20D45. Secondary: 12E05, 12E20, 15A21. \newline \emph{Key words and phrases:} Finite groups, Polycyclic groups, Computational group theory, Group automorphisms}}

\date{\today}

\maketitle

\abstract{Let $G$ be a finite solvable group, given through a refined consistent polycyclic presentation, and $\alpha$ an automorphism of $G$, given through its images of the generators of $G$. In this paper, we discuss algorithms for computing the order of $\alpha$ as well as the cycle length of a given element of $G$ under $\alpha$. We give correctness proofs and discuss the theoretical complexity of these algorithms. Along the way, we carry out detailed complexity analyses of several classical algorithms on finite polycyclic groups.}

\section{Introduction}\label{sec1}

\subsection{Background and aim of the paper}\label{subsec1P1}

The theory of polycyclic groups is a powerful tool for designing efficient algorithms for many computational problems on finite solvable groups, see \cite[Chapter 8]{HEO05a} for an introduction. One of these is an algorithm, originally described by Robinson in \cite{Rob81a} and later implemented and studied in detail by Smith in \cite{Smi94a}, for computing (generators of) the automorphism group of a finite solvable group $G$, assumed to be given as a \emph{(finite) pc group}, i.e., through a refined consistent polycyclic presentation (see \cite[Definitions 8.7, 8.10 and 8.18]{HEO05a} for the precise meaning of this, and \cite[Section 8.9]{HEO05a} for a concise overview of the main ideas on which the algorithm is based), thus providing a basis for the computational study of $\Aut(G)$ as an abstract group as well as of its natural action on $G$. However, it is not immediately clear which computational problems concerning $\Aut(G)$ can be solved efficiently on this basis, particularly since $\Aut(G)$ is in general not solvable. The aim of this paper is two-fold:

\begin{enumerate}
\item to discuss natural algorithms for the basic tasks of computing the orders of elements of the group $\Aut(G)$ (which are assumed to be given through their images of the presentation generators of $G$, as is the case by default in GAP \cite{GAP4}) and of determining the cycle length of a given element of $G$ under a given automorphism of $G$. We will prove the correctness of these algorithms (Theorem \ref{mainTheo1} and its proof in Section \ref{sec2}) and provide a theoretical complexity analysis for them (Theorem \ref{mainTheo2} and its proof in Section \ref{sec3}). We note that these algorithms have been implemented by the author in GAP, and the corresponding GAP source code is available from the author's website under \url{https://alexanderbors.wordpress.com/sourcecode/pcautord/}.
\item to give (in Subsection \ref{subsec3P1}) a detailed complexity analysis for several classical algorithms on finite pc groups, most notably (in the form of Theorem \ref{mainTheo3}) of a slightly modified version of an algorithm, due to Cannon, Eick and Leedham-Green \cite[Subsection 3.1]{CEL04a}, for computing a polycyclic generating sequence (pcgs) of $G$ that refines (in the precise sense of \cite[beginning of Section 2, p.~1446]{CEL04a}) the so-called LG-series of $G$, a characteristic series in $G$ with elementary abelian factors (see \cite[Subsection 2.1]{CEL04a}, where this series is called the \emph{elementary abelian nilpotent-central series} instead). This is useful for many applications, including our algorithms, but to the author's knowledge, there are no published results on the complexity of this algorithm. In fact, it seems that many important algorithms on pc groups currently lack published complexity analyses (quite contrarily to algorithms on permutation groups, for which Seress' book \cite{Ser03a} is a rich source of detailed complexity discussions), and we hope that our Theorem \ref{mainTheo3} and the auxiliary results from Subsection \ref{subsec3P1} will make such analyses more comfortable to do in the future.
\end{enumerate}

\subsection{Main results}\label{subsec1P2}

In this subsection, we state the main results of this paper in the form of Theorems \ref{mainTheo1}, \ref{mainTheo2} and \ref{mainTheo3} below. Theorem \ref{mainTheo1} is concerned with deterministic versions and Theorem \ref{mainTheo2} with probabilistic (Las Vegas) versions of the algorithms for computing automorphism orders resp.~cycle lengths that were mentioned in the first enumeration point in Subsection \ref{subsec1P1}; we give these algorithms in Subsection \ref{subsec1P4} below in pseudocode as Algorithms \ref{algo1} and \ref{algo2}. Actually, Algorithm \ref{algo2} is a bit more general than that, as it serves to compute cycle lengths of \emph{bijective affine maps on $G$} (functions on $G$ of the form $\A_{t,\alpha}:x\mapsto t\alpha(x)$ for a fixed $t\in G$ and $\alpha\in\Aut(G)$; with $t:=1$, this includes all automorphisms of $G$). We will see in Subsection \ref{subsec2P2} why it is natural to work with this larger class of functions.

\begin{theoremm}\label{mainTheo1}
Algorithms \ref{algo1} and \ref{algo2}, viewed as deterministic algorithms, are correct, i.e., they terminate on each input of the indicated form with the asserted output.
\end{theoremm}

\begin{theoremm}\label{mainTheo2}
Algorithms \ref{algo1} and \ref{algo2}, viewed as Las Vegas algorithms, both have expected running time subexponential in the input length.
\end{theoremm}

The following two remarks indicate that it is probably very difficult to improve Theorem \ref{mainTheo2}:

\begin{enumerate}
\item A fundamental obstacle to improving Theorem \ref{mainTheo2} by replacing \enquote{subexponential} by \enquote{polynomial} is the fact that all known algorithms for multiplication of elements of a finite polycyclic group, written in collected form (see \cite[Definition 8.13]{HEO05a}), have superpolynomial complexity (see below for a brief overview of results on the complexity of multiplication in (finite) polycyclic groups).
\item Similarly, there are fundamental obstacles to replacing the \enquote{Las Vegas} in Theorem \ref{mainTheo2} by \enquote{deterministic} (and deleting the word \enquote{expected}, of course), as not even for the special case of computing orders of invertible matrices over finite fields, any deterministic algorithms with subexponential worst-case complexity are known (see \cite{CL97a} for such an algorithm, which we will come back to later and which only fails to achieve polynomial complexity because of the seemingly inevitable use of integer factorization).
\end{enumerate}

When $P$ is a refined consistent polycyclic presentation representing the finite solvable group $G$, we denote by $\ell(P)$ the length of $P$ as an algorithm input (see also Subsection \ref{subsec1P3}) and by $\Coll(P)=\Coll_{\Acal}(P)$ the worst-case runtime (in bit operations) of a fixed algorithm $\Acal$ which takes as input the presentation $P$ and two elements of $G$ given in collected form with respect to $P$ and outputs the collected form of the product of these two elements. Throughout the rest of this paper, we assume that $\Acal$ has been fixed and suppress it in the notation $\Coll(P)$. As we will incorporate $\Coll(P)$ in our complexity bounds, we will not be concerned with \enquote{good} choices for $\Acal$ here, but naturally, finding efficient multiplication algorithms is a problem of fundamental importance in computational group theory, and especially for polycyclic groups, it has been studied for a long time by various authors; let us give an overview.

Multiplication of elements of polycyclic groups in collected form is always carried out using some form of \emph{collection} (whence our notation $\Coll(P)$), i.e., replacing minimal non-collected subwords of the concatenation of the two collected forms using the defining relations of the associated polycyclic presentation until one arrives at a word in collected form. The art of making the computations efficient lies in how to choose the next non-collected subword to process at any given step. After earlier approaches following Hall's pioneering theoretical work \cite{Hal34a}, which involved \enquote{collection from the left} with a constraint on the generators involved in the subword, and the discovery, originally due to Neub{\"u}ser, that \enquote{collection from the right} without such a constraint was more efficient (and such algorithms had been implemented by Felsch \cite{Fel76a} as well as Havas and Nicholson \cite{HN76a}), nowadays, it is a common agreement that \enquote{constraintless collection from the left}, henceforth simply \enquote{collection from the left}, is, in practice and on average, the most efficient approach. This convention was preceded by the two fundamental papers \cite{Vau90a} and \cite{LS90a}, due to Vaughan-Lee resp.~Leedham-Green and Soicher, which in combination gave both practical and theoretical evidence that \enquote{collection from the left} is superior to \enquote{collection from the right}. Later work focused on refined variants of \enquote{collection from the left}, usually involving additional assumptions on the presentation; we mention Leedham-Green and Soicher's \enquote{Deep Thought} algorithm \cite{LS98a}, which works for certain polycyclic presentations of nilpotent groups, H{\"o}fling's unpublished preprint \cite{Hof04a}, where the discussed algorithm involves passing from an arbitrary to a \enquote{nice} polycyclic presentation (of an arbitrary finite solvable group), and the relatively recent paper \cite{NN15a} by Newman and Niemeyer, which derives a nice, compact upper bound on the complexity of \enquote{collection from the left} for certain polycyclic presentations (which are such that, as is also argued there \cite[Lemma 1]{NN15a}, each finite solvable group has such a presentation).

As mentioned in Subsection \ref{subsec1P1} already, the following complexity result of independent interest is also important for the proof of Theorem \ref{mainTheo2}.

\begin{theoremm}\label{mainTheo3}
Let $G$ be a finite solvable group, given via a refined consistent polycyclic presentation $P$. One can compute in $O(\ell(P)^8\Coll(P)+\ell(P)^{10})$ bit operations

\begin{itemize}
\item a pc group isomorphism (see the end of Subsection \ref{subsec1P3}) from $P$ to another refined consistent polycyclic presentation $\tilde{P}$ of $G$, yielding a pcgs $\vec{g}$ of $G$ with associated presentation $\tilde{P}$ and whose entries are given in collected form with respect to the generators of $P$, such that $\vec{g}$ refines (in the sense of \cite[beginning of Section 2, p.~1446]{CEL04a}) the LG-series of $G$ (see \cite[Subsection 2.1]{CEL04a}), and
\item the sequence of \emph{final weights} (see \cite[Subsubsection 3.1.1]{CEL04a}) of the entries of $\vec{g}$ with respect to the LG-series of $G$.
\end{itemize}
\end{theoremm}

As $\Coll(P)$ can be made subexponential in $\ell(P)$ through a suitable choice of $\Acal$ by the results discussed above, this shows in particular that the complexity of Step 1 in the two algorithms is subexponential in the input length. The algorithm which we will analyze for proving Theorem \ref{mainTheo3} is essentially the original one from \cite[Subsection 3.1]{CEL04a}; however, a slight modification will be necessary at some point, see the paragraph before Lemma \ref{modifyPcgsLem}.

\subsection{Notation and terminology}\label{subsec1P3}

We denote by $\IN$ the set of natural numbers (including $0$) and by $\IN^+$ the set of positive integers. For a function $f$ and a set $X$, the element-wise image of $X$ under $f$ is denoted by $f[X]$, and the restriction of $f$ to $X$ by $f_{\mid X}$. The identity function on a set $X$ is denoted by $\id_X$. For a positive real number $c$ with $c\not=1$, we denote by $\log_c$ the base $c$ logarithm function, and $\log:=\log_{\e}$ denotes the natural logarithm function (with base the Euler constant $\e$). For a prime power $q$, the finite field with $q$ elements is denoted by $\IF_q$. For a prime element $p$ of a factorial ring $R$ and $x\in R$, we denote by $\nu_p(x)$ the $p$-adic valuation of $x$ (the largest non-negative integer $v$ such that $p^v$ divides $a$, understood to be $\infty$ if $x=0$). The exponent of a finite group $G$ is denoted by $\Exp(G)$, and the order of an element $g\in G$ by $\ord(g)$; this also applies to automorphisms of $G$, viewed as elements of the group $\Aut(G)$. For an element $t$ and an automorphism $\alpha$ of a group $G$, we denote by $\A_{t,\alpha}$ the bijective affine map $G\rightarrow G, x\mapsto t\alpha(x)$. At several points, will use the Kronecker delta $\delta_{x,y}$, which is defined to be $1$ (i.e., the integer $1$ or more generally the unity element of a ring, depending on the context) if $x=y$, and $0$ otherwise. Throughout the paper, we will be using much of the terminology from \cite{CEL04a}; in particular, by a \emph{pcgs} of a finite solvable group $G$, we always mean a polycyclic generating sequence refining some composition series of $G$, i.e., such that all relative orders (in the sense of \cite[Definition 8.2]{HEO05a}) of the pcgs entries are primes (and thus the associated (consistent) polycyclic presentation of $G$ is refined). If $\vec{g}=(g_1,\ldots,g_n)$ is a pcgs of length $n$ of a polycyclic group $G$, then for $h\in G$ and $k\in\{1,\ldots,n\}$, we denote by $\exp_{\vec{g}}(h,k)$ the unique element of $\{0,1,\ldots,p_k-1\}$, where $p_k$ denotes the relative order of $g_k$, such that
\begin{equation}\label{notationEq}
h=g_1^{\exp_{\vec{g}}(h,1)}g_2^{\exp_{\vec{g}}(h,2)}\cdots g_n^{\exp_{\vec{g}}(h,n)}.
\end{equation}
The right-hand side of Formula (\ref{notationEq}), viewed as a \enquote{formal expression} (more precisely, a product of powers of the variables $g_1,\ldots,g_n$ where the exponents are given through their binary digit representations) will be called the \emph{$\vec{g}$-collected form of $g$}, and the tuple $\exp_{\vec{g}}(h):=(\exp_{\vec{g}}(h,1),\exp_{\vec{g}}(h,2),\ldots,\exp_{\vec{g}}(h,n))$ will be called the \emph{$\vec{g}$-exponent vector of $h$} (see also \cite[Definition 8.4]{HEO05a}); these two notions are basically interchangeable, and while the collected form representation is less compact than the exponent vector representation, we will use the former frequently in this paper for better readability. Also, when a finite solvable group $G$ is given through a refined consistent polycyclic presentation $\langle X\mid R\rangle$, then we may speak of the \emph{$X$-collected forms} or \emph{$X$-exponent vectors} of the elements of $G$; whenever we do so, we have in mind the representation of $G$ as the quotient $\operatorname{F}(X)/\langle\langle R\rangle\rangle$ of the free group on $X$ by the normal subgroup generated by the relators in $R$, and we identify the entries of $X$ with their images under the canonical projection $\operatorname{F}(X)\rightarrow\operatorname{F}(X)/\langle\langle R\rangle\rangle=G$, so that we can view $X$ as an actual pcgs of $G$ and this terminology makes sense. Finally, if $\alpha$ is an automorphism of a finite solvable group $G$ and $\vec{g}=(g_1,\ldots,g_n)$ is a pcgs of $G$, then the \emph{$\vec{g}$-collected form of $\alpha$} (resp.~the \emph{$\vec{g}$-exponent matrix of $\alpha$}) is the $n$-tuple consisting of the $\vec{g}$-collected forms (resp.~the $\vec{g}$-exponent vectors) of the images $\alpha(g_i)$.

As for our computational model, we use the same assumptions as in \cite[Section 2]{Hof04a}. In particular, considering the numbers $\ell(P)$ and $\Coll(P)$ associated with the refined consistent polycyclic presentation $P$ and introduced in Subsection \ref{subsec1P2}, we assume that there are positive constants $c,d$ such that $c\log{|G|}\leq \ell(P)\leq(\log{|G|})^d$ and that $\ell(P)\leq\Coll(P)\leq f(\ell(P))$ for some function $f$ of subexponential growth. We also denote by $d(P)$ the number of presentation generators of $P$, which equals the composition length of the abstract group $G$ represented by $P$, by $e(P)$ the maximum binary representation length of one of the relative orders of the presentation generators of $P$ (i.e., of one of the prime divisors of $|G|$), and we assume that $\max\{d(P),e(P)\}\leq \ell(P)$. Note also that $d(P)=d(Q)$ and $e(P)=e(Q)$ for any other refined consistent polycyclic presentation $Q$ of the same abstract group $G$. If $Q$ is another refined consistent polycyclic presentation of $G$, then a \emph{pc group isomorphism} $P\rightarrow Q$ consists of two tuples whose entries are, for fixed polycyclic generating sequences $\pi_P$ resp.~$\pi_Q$ of $G$ with associated presentation $P$ resp.~$Q$, the $\pi_Q$-exponent vectors of the elements of $\pi_P$, resp.~the $\pi_P$-exponent vectors of the elements of $\pi_Q$.

\subsection{Our two algorithms}\label{subsec1P4}

Below, we give the two algorithms for automorphism order resp.~affine map cycle length computation in pseudocode. For better readability, we will use the abbreviation $L(i):=\sum_{t=1}^i{\ell_t}$. At the moment, we do not explain the ideas behind the single steps of the algorithms (since this would require some theoretical results discussed in Subsection \ref{subsec2P2}), but they will become clear during the correctness proofs in Subsections \ref{subsec2P3} and \ref{subsec2P4}.

\begin{algorithm}
\SetKwInOut{Input}{input}\SetKwInOut{Output}{output}

\Input{A finite solvable group $G$, given through a refined consistent polycyclic presentation $\langle X\mid R\rangle$, and (the $X$-exponent matrix of) an automorphism $\alpha$ of $G$.}
\Output{The order of $\alpha$, i.e., the least common multiple of the cycle lengths of the permutation $\alpha$ on $G$.}
\BlankLine
\nl Compute the following (see \cite[Subsection 3.1]{CEL04a} and our Subsection \ref{subsec3P1} for details):
\begin{itemize}
\item another refined consistent polycylic presentation $\langle Y\mid S\rangle$, with $Y=(y_1,\ldots,y_n)$, associated with some other pcgs $\vec{g}=(g_1,\ldots,g_n)$ of $G$, which refines (in the sense of \cite[beginning of Section 2, p.~1446]{CEL04a}) the \emph{elementary abelian nilpotent-central series of $G$} (also called the \emph{LG-series of $G$}; see \cite[Subsection 2.1, p.~1447]{CEL04a} for its definition) $G=G_1>G_2>\cdots>G_r>G_{r+1}=\{1\}$.
\item the sequence of \emph{final weights} (see \cite[Subsubsection 3.1.1, p.~1450]{CEL04a}) of the $g_i$ with respect to the LG-series of $G$.
\item the $\vec{g}$-exponent matrix of $\alpha$.
\end{itemize}

\nl \For{$i=1,\ldots,r$}{
\nl Set $\ell_i$ to be the number of entries of $\vec{g}$ with final weight $i$.

\nl Set $p_i$ to be the common relative order of the pcgs entries $g_j$ with $j=L(i-1)+1,L(i-1)+2,\ldots,L(i)$.
}

\nl \For{$i=1,\ldots,r$}{
\nl Set $M_i=(a(i,j,k))_{j,k=1}^{\ell_i}$ to be the (invertible) $(\ell_i\times \ell_i)$-matrix over $\IF_{p_i}=\IZ/p_i\IZ=\{\overline{0},\ldots,\overline{p_i-1}\}$ such that $a(i,j,k)=\overline{\exp_{\vec{g}}(\alpha(g_{L(i-1)+k}),L(i-1)+j)}$.

\nl Set $o_i$ to be the order of $M_i\in\GL_{\ell_i}(p_i)$, computed as described in Subsection \ref{subsec2P1}.
}

\nl Set $o:=\lcm_{i=1,\ldots,r}{o_i}$.

\nl Compute the $\vec{g}$-exponent matrix of $\beta:=\alpha^o$.

\nl \For{$i=1,\ldots,r-1$}{
\nl \If{at least one of the numbers $\exp_{\vec{g}}(\beta(g_j),k)$ for $j=1,2,\ldots,L(i)$ and $k=L(i)+1,L(i)+2,\ldots,L(i+1)$ is nonzero}{
\nl Replace $o$ by $o\cdot p_{i+1}$.

\nl Replace (the $\vec{g}$-exponent matrix of) $\beta$ by (the $\vec{g}$-exponent matrix of) $\beta^{p_{i+1}}$.
}
}

\nl Return $o$.
\caption{Automorphism order computation}\label{algo1}
\end{algorithm}

\begin{algorithm}
\SetKwInOut{Input}{input}\SetKwInOut{Output}{output}

\Input{A finite solvable group $G$, given through a refined consistent polycyclic presentation $\langle X\mid R\rangle$, (the $X$-exponent vectors of) elements $g,t\in G$ and (the $X$-exponent matrix of) an automorphism $\alpha$ of $G$.}
\Output{The cycle length of $g$ under $\A_{t,\alpha}$.}

\nl Compute the following (see \cite[Subsection 3.1]{CEL04a} and our Subsection \ref{subsec3P1} for details):
\begin{itemize}
\item another refined consistent polycylic presentation $\langle Y\mid S\rangle$, with $Y=(y_1,\ldots,y_n)$, associated with some other pcgs $\vec{g}=(g_1,\ldots,g_n)$ of $G$, which refines (in the sense of \cite[beginning of Section 2, p.~1446]{CEL04a}) the \emph{elementary abelian nilpotent-central series of $G$} (also called the \emph{LG-series of $G$}; see \cite[Subsection 2.1, p.~1447]{CEL04a} for its definition) $G=G_1>G_2>\cdots>G_r>G_{r+1}=\{1\}$.
\item the sequence of \emph{final weights} (see \cite[Subsubsection 3.1.1, p.~1450]{CEL04a}) of the $g_i$ with respect to the LG-series of $G$.
\item the $\vec{g}$-exponent matrix of $\alpha$, and the $\vec{g}$-exponent vectors of $t$ and $g$.
\end{itemize}

\nl \For{$i=1,\ldots,r$}{
\nl Set $\ell_i$ to be the number of entries of $\vec{g}$ with final weight $i$.

\nl Set $p_i$ to be the common relative order of the pcgs entries $g_j$ with $j=L(i-1)+1,L(i-1)+2,\ldots,L(i)$.
}

\nl Set $\lambda:=1$.

\nl \For{$i=1,\ldots,r$}{
\nl Set $M_i$ to be the (invertible) $(\ell_i\times \ell_i)$-matrix $(a(i,j,k))_{j,k=1}^{l_i}$ over $\IF_{p_i}=\IZ/p_i\IZ=\{\overline{0},\ldots,\overline{p_i-1}\}$ such that $a(i,j,k)=\overline{\exp_{\vec{g}}(\alpha(g_{L(i-1)+k}),L(i-1)+j)}$.

\nl Set $u_i$ to be the $\ell_i$-dimensional vector $(x(i,j))_{j=1}^{\ell_i}$ over $\IF_{p_i}$ such that $x(i,j)=\overline{\exp_{\vec{g}}(g^{-1}t\cdot\alpha(g),L(i-1)+j)}$.

\nl Set $\lambda_i$ to be the cycle length of the zero vector in $\IF_{p_i}^{\ell_i}$ under the bijective affine transformation $v\mapsto M_iv+u_i$, computed as described in Subsection \ref{subsec2P1}.

\nl Replace $\lambda$ by $\lambda\cdot\lambda_i$.

\nl Replace $\A_{t,\alpha}$ by $\A_{t,\alpha}^{\lambda_i}$ (that is, replace (the $\vec{g}$-collected forms of) $t$ and $\alpha$ accordingly).
}

\nl Return $\lambda$.
\caption{Affine map cycle length computation}\label{algo2}
\end{algorithm}

\section{Details on and correctness proofs for Algorithms \ref{algo1} and \ref{algo2}}\label{sec2}

\subsection{Details on computing orders and cycle lengths in the elementary abelian case}\label{subsec2P1}

In this subsection, we give details on how we intend to perform the computation of the order of the matrix $M_i$ in Step 7 of Algorithm \ref{algo1} and of the cycle length $\lambda_i$ in Step 9 of Algorithm \ref{algo2}, both deterministically and probabilistically. This serves two purposes: Firstly, to remove ambiguity from the pseudocode formulation, and secondly, to prepare for the complexity analysis in Subsection \ref{subsec3P2}. We also include a few comments on the author's GAP implementations of these algorithms where appropriate.

In general, for computing the order of an invertible matrix $M$ over a finite prime field $\IF_p$, proceed as in \cite{CL97a}. Note that this requires us to factor integers of the form $p^d-1$ with $d$ at most the dimension of $M$; in the probabilistic version of Algorithm \ref{algo1}, with whose complexity we are concerned in Theorem \ref{mainTheo2}, we will assume that a combination of the AKS algorithm (see \cite{AKS04a}), for deterministic primality testing in polynomial time, and the general number field sieve (see \cite{Pom96a}), for Las Vegas factorization of non-prime integers (or rather, only of integers that are not prime powers, but that is not a problem) in subexponential expected time, is used for this, whereas in the deterministic version, where we are not concerned with complexity issues, we may use any of the known deterministic integer factorization algorithms (the one from \cite{Hit18a} has the currently best worst-case complexity).

It remains to discuss how to compute, for a given prime $p$, positive integer $d$, invertible $(d\times d)$-matrix $M$ over $\IF_p$ and vector $t\in\IF_p^d$, the cycle length of the zero vector of $V=\IF_p^d$ under the bijective affine transformation $\A_{t,M}:V\rightarrow V,v\mapsto Mv+t$.

\begin{enumerate}
\item Use any of the available polynomial-time algorithms to compute a basis change matrix $T$ such that $T^{-1}MT$ is in rational canonical form. One example of such an algorithm is the basic one described in \cite[Section 12.2, pp.~481f.]{DF04a} (in the author's GAP implementations of Algorithms \ref{algo1} and \ref{algo2}, an implementation of that algorithm by Hulpke is used); a more efficient variant is discussed in \cite[Chapter 9]{Sto00a}.
\item The cycle length of the zero vector under $\A_{t,M}$ then is the same as its cycle length under $\A_{T^{-1}t,T^{-1}MT}$, which in turn is the least common multiple of the cycle lengths of the zero vector in the subspaces of $V$ corresponding to the companion matrix blocks of $T^{-1}MT$ under the respective restrictions of $\A_{T^{-1}t,T^{-1}MT}$. It therefore suffices to consider the case where $M$ is the companion matrix of a monic polynomial $P\in\IF_p[x]$, with $x$ being a variable.
\item In that case, the action of $M$ on $\IF_p^d$ is isomorphic to the one of the multiplication by $x$ modulo $P$ on the quotient algebra $\IF_p[x]/(P)$, with an isomorphism given by the map $\pi_P:(a_1,\ldots,a_d)\mapsto\sum_{\ell=0}^{d-1}{a_{\ell+1}x^{\ell}}+(P)$. Hence we will actually compute the cycle length of the zero element $(P)\in\IF_p[x]/(P)$ under the bijective affine map $F+(P)\mapsto \pi_P(t)+xF+(P)$ on $\IF_p[x]/(P)$.
\item To this end, factor $P$ in $\IF_p[x]$, yielding $P=\prod_{j=1}^r{Q_j^{m_j}}$, where the $Q_j\in\IF_p[x]$ are irreducible and pairwise distinct. For the sake of unambiguousness and to make the algorithms deterministic, say we use Berlekamp's algorithm from \cite{Ber67a} for this; for the later theoretical complexity analysis, we will instead assume that Berlekamp's Las Vegas factorization algorithm with expected running time polynomial in the input size (see \cite{Ber70a} and \cite[Section 3]{GP01a}) is used. Now set $R_j:=Q_j^{m_j}$ and, writing $t=(t_1,\ldots,t_d)$, set $S:=\sum_{\ell=0}^{d-1}{t_{\ell+1}x^{\ell}}\in\IF_p[x]$. Then by the Chinese Remainder Theorem, the cycle length of the zero vector under $A$ is the least common multiple of the cycle lengths of the zero vectors in the quotients $\IF_p[x]/(R_j)$ under $F+(R_j)\mapsto S+xF+(R_j)$, for $j=1,\ldots,r$. It therefore suffices to consider the case where $P=Q^m$ is a power of an irreducible polynomial $Q\in\IF_p[x]$.
\item In that case, if $S=0$, then the cycle length is $1$, so assume that $S\not=0$. Consider the affine map $A:F\mapsto x\cdot F+S$ on $\IF_p[x]$. By induction on $n\in\IN$, it is easy to show that $A^n(0)=S\cdot(1+x+x^2+\cdots+x^{n-1})$, so the cycle length on the quotient algebra $\IF_p[x]/(Q^m)$ in which we are interested is the smallest $n\in\IN^+$ such that
\[
S(1+x+x^2+\cdots+x^{n-1})\equiv 0\Mod{Q^m},
\]
which, by multiplying both congruence sides by $x-1$, is equivalent to
\begin{equation}\label{congruence}
S\cdot(x^n-1)\equiv 0\Mod{Q^{m+\delta_{Q,x-1}}}.
\end{equation}
Set $v:=\nu_Q(S)$, the $Q$-adic valuation of $S$ (see Subsection \ref{subsec1P3}), and note that $0\leq v<m$. Formula (\ref{congruence}) is equivalent to
\[
x^n-1\equiv 0\Mod{Q^{m-v+\delta_{Q,x-1}}},
\]
and the smallest such $n\in\IN^+$ is by definition just the order of the polynomial $Q^{m-v+\delta_{Q,x-1}}$, which we compute following \cite[Theorems 3.8 and 3.9 and p.~87, below Theorem 3.11]{LN97a}; note that this requires us to factor the integer $p^{\deg{Q}}-1$, which we do as described above.
\end{enumerate}

We note that similarly to the approach in \cite{CL97a}, one could speed Algorithm \ref{algo2} up a bit by not fully factoring the polynomial $P$ in point (4), but only working with the squarefree factorization of $P$. This does, however, \emph{not} improve the asymptotic complexity of the Las Vegas version of Algorithm \ref{algo2}, as one still needs to carry out the mentioned integer factorizations. In the author's GAP implementation of Algorithm \ref{algo2}, the full factorization of $P$ is used.

\subsection{Automorphisms restricted to cosets}\label{subsec2P2}

The following result, which is essentially \cite[Lemma 2.1.3(1)]{Bor16a}, will be used in the correctness proofs for both algorithms:

\begin{lemmma}\label{transferLem}
Let $G$ be a finite group, $t\in G$, $\alpha$ an automorphism of $G$, $H$ an $\alpha$-invariant subgroup of $G$, and set $A:=\A_{t,\alpha}$. Then if $xH$ is a left coset of $H$ in $G$ such that $A[xH]=xH$, and writing $A(x)=xh_0$ with $h_0\in H$, then for all $h\in H$, $A(xh)=xh_0\alpha(h)=x\A_{h_0,\alpha_{\mid H}}(h)$. In particular, the cycle length of $x$ under $A$ then equals the cycle length of $1_G=1_H$ under the bijective affine map $\A_{h_0,\alpha_{\mid H}}$ on $H$.\qed
\end{lemmma}

In other words, in the setting of Lemma \ref{transferLem}, the action of $A$ on the coset $xH$ is isomorphic to the action of $\A_{h_0,\alpha_{\mid H}}$ on $H$. As noted in \cite[Remark 2.1.4(1)]{Bor16a} already, even if $t=1$, i.e., if $A=\alpha$ is an automorphism of $G$, the coset representative $x$ cannot necessarily be chosen such that $h_0=x^{-1}\alpha(x)=1$, so the affine map on $H$ describing the action of $\alpha$ on $xH$ is in general still only an affine map (not an automorphism) on $H$, which is why it is more natural to work with affine maps in situations where Lemma \ref{transferLem} is used.

We also note the following consequence of Lemma \ref{transferLem} (see also \cite[proof of Theorem 2]{Hor74a}), which will be used in the correctness proof of Algorithm \ref{algo1}:

\begin{lemmma}\label{idIdLem}
Let $G$ be a finite group, $\alpha$ an automorphism of $G$, $p$ a prime, $N$ an $\alpha$-invariant normal subgroup of $G$ with $\Exp(N)=p$. Assume that the restriction $\alpha_{\mid N}$ and the automorphism $\tilde{\alpha}$ of $G/N$ induced by $\alpha$ are the identity on $N$ and $G/N$ respectively. Then either $\alpha=\id_G$ or $\alpha$ is of order $p$.
\end{lemmma}

\begin{proof}
By the assumption that $\tilde{\alpha}=\id_{G/N}$, $\alpha$ restricts to a permutation on each coset $xN$ of $N$ in $G$, and by the assumption that $\alpha_{\mid N}=\id_N$ and Lemma \ref{transferLem}, this permutation on $xN$ is isomorphic to the left translation by a fixed element on $N$. Hence, as $N$ has exponent $p$, each such restriction of $\alpha$ is either trivial or has order $p$, and so $\alpha$ as a whole is either trivial or has order $p$.
\end{proof}

\subsection{Correctness of Algorithm \ref{algo1}}\label{subsec2P3}

As in the description of Algorithm \ref{algo1}, let $G=G_1>G_2>\cdots>G_r>G_{r+1}=\{1\}$ denote the LG-series of $G$, and for $i=1,\ldots,r$, denote by $V_i:=G_i/G_{i+1}$ the $i$-th factor in the series, a finite elementary abelian group. Then $\ell_i$, defined in Step 3, is the length of the pcgs/basis of $V_i$ induced by $\vec{g}$, so $\ell_i$ is the dimension of the vector space $V_i$, and $p_i$, defined in Step 4, is the exponent of $V_i$ (i.e., the characteristic/cardinality of the associated finite prime field).

The matrix $M_i$, defined in Step 6, represents the automorphism $\alpha_i$ of $V_i$ induced by $\alpha$ with respect to the $\IF_{p_i}$-basis
\[
g_{L(i-1)+1}G_{i+1},g_{L(i-1)+2}G_{i+1},\ldots,g_{L(i)}G_{i+1}
\]
of $V_i$. For each $i=1,\ldots,r$, the function $\Aut(G)\rightarrow\Aut(V_i)$ mapping an automorphism of $G$ to the corresponding induced automorphism of $V_i$ is a group homomorphism, and so each $o_i$ (defined in Step 7) divides $\ord(\alpha)$, whence $o$ (defined in Step 8) also divides $\ord(\alpha)$, and thus $\ord(\alpha)=o\cdot\ord(\alpha^o)$.

The last for-loop (ranging from Steps 10--13) serves to compute $\ord(\alpha^o)$ (with $o$ as defined in Step 8) step by step, updating the values of the variables $o$ and $\beta$ along the way so that at the end of the loop, the value of $o$ will be $\ord(\alpha)$. At the beginning of the $i$-th step of the loop, the value of $\beta$ is $\alpha^o$, and the value of $o$ is a divisor of $\ord(\alpha)$ such that the automorphisms of $G/G_{i+1}$ and $G_{i+1}/G_{i+2}$ respectively which are induced by $\beta=\alpha^o$ are both trivial (actually, by construction, for each $j\in\{1,\ldots,r\}$, the automorphism of $G_j/G_{j+1}$ induced by $\beta$ is trivial). Then precisely one of the following two cases occurs:

\begin{itemize}
\item $\beta$ is also trivial modulo $G_{i+2}$, so that neither of the two variables needs to be updated for the next loop step.
\item $\beta$ is nontrivial modulo $G_{i+2}$. In that case, by Lemma \ref{idIdLem}, applied to the group $G/G_{i+2}$, the order of the automorphism of $G/G_{i+2}$ induced by $\beta$ is equal to $\Exp(G_{i+1}/G_{i+2})=\Exp(V_{i+1})=p_{i+1}$, so that for the next loop step, $o$ must be replaced by $o\cdot p_{i+1}$ and $\beta$ by $\beta^{p_{i+1}}$.
\end{itemize}

The if clause in Step 11 tests whether the second case occurs and if so, updates $o$ and $\beta$ as described above. Note that by assumption, each of the generators $g_{L(i)+1},g_{L(i)+2},\ldots,g_{L(i+1)}$ is fixed modulo $G_{i+2}$ by $\beta$, which explains the range for $j$ in the if clause, and that for $k=1,2,\ldots,L(i)$, $\exp_{\vec{g}}(\beta(g_j),k)$ equals $\delta_{k,j}$ by assumption, whereas the values of $\exp_{\vec{g}}(\beta(g_j),k)$ for $k>L(i+1)$ do not matter for the question whether $\beta$ is trivial modulo $G_{i+2}$, which explains the range for $k$.

\subsection{Correctness of Algorithm \ref{algo2}}\label{subsec2P4}

We use the notation from the first paragraph of the correctness proof for Algorithm \ref{algo1} and note that the statements on $\ell_i$ and $p_i$ from there apply here as well. Throughout, we assume that the variable $A$ stands for the bijective affine map $\A_{t,\alpha}$ on $G$ where $t$ and $\alpha$ are as in the input. The values of the variables $t$ and $\alpha$ themselves will, however, be changed along the way.

At the beginning of the $i$-th step of the last for-loop (Steps 6--11), $\lambda$ is the smallest divisor of the cycle length of $g$ under $A$ such that $g$ is a fixed point of $A^{\lambda}$ modulo $G_i$ (i.e., $A^{\lambda}(g)=g\cdot g_i$ for some $g_i\in G_i$), and $t$ and $\alpha$ are such that $A^{\lambda}=\A_{t,\alpha}$. At the end of the loop step, we want the analogous situation with $i$ replaced by $i+1$; that is, we want to find the smallest $\lambda_i\in\IN^+$ such that $A^{\lambda\cdot\lambda_i}(g)=g\cdot g_{i+1}$ for some $g_{i+1}\in G_{i+1}$.

To this end, we use Lemma \ref{transferLem} to translate the action of $A^{\lambda}=\A_{t,\alpha}$ on $gG_i$ into the action of $B_i:=\A_{g_i,\alpha_{\mid G_i}}$ on $G_i$. Note that since the isomorphism transforming the two actions into each other is just a left translation by a fixed element in both directions (namely by $g$ respectively $g^{-1}$), our problem is equivalent to finding the smallest $\lambda_i\in\IN^+$ such that the cycle of $1\in G_i$ under $B_i$ is of length $\lambda_i$ modulo $G_{i+1}$, so we study the induced action of $B_i$ on $V_i=G_i/G_{i+1}$.

As in Algorithm \ref{algo1}, the matrix $M_i$ from Step 7 represents the induced action of $\alpha$ on $V_i$, and since $g_i=g^{-1}\cdot\A_{t,\alpha}(g)$ by definition, the vector $u_i$ from Step 8 is the projection of $g_i$ to $V_i$, so the number $\lambda_i$, which we want to compute, is just the cycle length of $0\in\IF_{p_i}^{\ell_i}$ under $v\mapsto M_iv+u_i$, whence Step 9.

So at the end of the last loop step (and thus in Step 12), $\lambda$ divides the cycle length of $g$ under $A$ while at the same time, $g$ is a fixed point of $A^{\lambda}$ modulo $G_{r+1}=\{1\}$, hence a fixed point, period. Therefore, $\lambda$ is indeed the cycle length of $g$ under $A$ and hence the correct output.

\section{Complexity analysis}\label{sec3}

\subsection{Complexity analysis for Theorem \ref{mainTheo3}}\label{subsec3P1}

The most demanding part of the complexity analysis of both our algorithms lies in Step 1, the passage to a \enquote{nicer} presentation of $G$. An algorithm for computing a pcgs refining the LG-series is described in \cite[Subsection 3.1]{CEL04a}. We will show that following that approach with some modifications yields the validity of Theorem \ref{mainTheo3}.

We will require several lemmas concerning the theoretical complexity of basic problems, such as computing powers of elements and of automorphisms of finite polycyclic groups.

\begin{lemmma}\label{squareMultiplyLem}
For every refined consistent polycyclic presentation $P=\langle X\mid R\rangle$ representing the (finite solvable) group $G$, the following hold:

\begin{enumerate}
\item For every $g\in G$ and all $e\in\IZ$, (the $X$-exponent vector of) the power $g^e$ can be computed from (the one of) $g$ using $O(d(P)+\log{|e|})$ multiplications of elements in $X$-collected form and $O(d(P)e(P)+\log{|e|})$ bit operations spent outside group element multiplications.
\item For all automorphisms $\alpha_1,\alpha_2\in\Aut(G)$ (given through their $X$-exponent matrices), the ($X$-exponent matrix of the) composition $\alpha_1\circ\alpha_2$ can be computed using $O(d(P)^3+d(P)^2e(P))$ multiplications of elements in $X$-collected form and $O(d(P)^3e(P))$ bit operations spent outside group element multiplications.
\item For every $\alpha\in\Aut(G)$ (given through its $X$-exponent matrix) and all $e\in\IN^+$, the ($X$-exponent matrix of the) iterate $\alpha^e$ can be computed using $O(\log{e}\cdot(d(P)^3+d(P)^2e(P)))$ multiplications of elements in $X$-collected form as well as $O(\log{e}\cdot d(P)^3e(P))$ bit operations spent outside group element multiplications.
\end{enumerate}
\end{lemmma}

\begin{proof}
For statement (1): If $e\geq 0$, use a square-and-multiply approach. First, compute and store $g,g^2,g^4,\ldots,g^{2^{\lfloor\log_2{e}\rfloor}}$; in each iteration step, the group multiplication algorithm is called once for squaring and one moves a marker along the given binary digit expansion of $e$ one step further so that one knows when to stop the iterated squaring. This requires $O(\log{e})$ multiplications and $O(\log{e})$ bit operations for other purposes (moving the marker). Afterward, read the digits of $e$ one after the other and multiply the corresponding powers of $g$ computed before, which also requires $O(\log{e})$ multiplications and $O(\log{e})$ bit operations outside multiplication.

If $e<0$, then first compute $g^{-1}$. Say the $X$-depth of $g$ (see \cite[Definition 8.5(a)]{HEO05a}) is $d$, and say the $X$-collected form of $g$ starts with the power $x_d^{e_d}$, where $e_d\in\{1,\ldots,p_d-1\}$ and $p_d$ is the relative order of $x_d$. Then the $X$-collected form of $g^{-1}$ starts with the power $x_d^{p_d-e_d}$, and $g\cdot x_d^{p_d-e_d}$ has depth strictly larger than $d$. It follows that the $X$-collected form of $g^{-1}$ can be computed in a recursive manner, using $O(d(P))$ multiplications of elements in $X$-collected form and $O(d(P)e(P))$ other bit operations (from $O(d(P))$ subtractions of positive integers with $O(e(P))$ binary digits each). After this, raise $g^{-1}$ to the $(-e)$-th power as in the previous paragraph, requiring another $O(\log{|e|})$ multiplications and other-purpose bit operations.

For statement (2): By assumption, we can read off $\alpha_2(x_i)=x_1^{e(i,1)}\cdots x_{d(P)}^{e(i,d(P))}$ in $X$-collected form directly from the input, for each $i=1,\ldots,d(P)$. Then
\[
(\alpha_1\circ\alpha_2)(x_i)=\alpha_1(x_1)^{e(i,1)}\cdots\alpha_1(x_{d(P)})^{e(i,d(P))},
\]
the $X$-collected form of which can be computed in view of statement (1) using
\[
O(d(P)\cdot(d(P)+\max\{\log{e(i,1)},\ldots,\log{e(i,d(P))}\})+d(P))\subseteq O(d(P)^2+d(P)e(P))
\]
multiplications and
\[
O(d(P)\cdot d(P)e(P))=O(d(P)^2e(P))
\]
other-purpose bit operations.

For statement (3): Similar to statement (1), using statement (2) for each squaring step and for the subsequent composition of suitable powers of the form $\alpha^{2^f}$.
\end{proof}

The next lemma discusses the complexity of transforming the polycyclic presentation and automorphism/group elements under an \enquote{elementary transformation step} of the associated pcgs; Step 1 in Algorithms \ref{algo1} and \ref{algo2} essentially consists of a sequence of applications of such elementary steps, similarly to \cite[Subsection 3.1]{CEL04a}.

\begin{lemmma}\label{elementaryTransfoLem}
For every refined consistent polycyclic presentation $P=\langle X\mid R\rangle$, representing the finite group $G$, with $X=(x_1,\ldots,x_n)$, and every $g\in G\setminus\{1\}$, say of $X$-depth $d$ (in the sense of \cite[Definition 8.5(a)]{HEO05a}), the following hold:

\begin{enumerate}
\item For every $h\in G$, the $X_g$-exponent vector of $h$, where $X_g:=(y_1,\ldots,y_n)$ with $y_i=x_i$ if $i\not=d$ and $y_d=g$, can be computed using
\[
O(d(P)^2+d(P)e(P))
\]
multiplications of elements in $X$-collected form and $O(d(P)^2e(P)+d(P)e(P)^3)$ bit operations spent for other purposes.
\item The refined consistent polycyclic presentation $P'$ of $G$ associated with $X_g$ can be computed using
\[
O(d(P)^4+d(P)^3e(P))
\]
multiplications of elements in $X$-collected form and $O(d(P)^4e(P)+d(P)^3e(P)^3)$ bit operations spent for other purposes.
\item A pc group isomorphism from $P$ to $P'$ (in the sense of Subsection \ref{subsec1P3}) can be computed using
\[
O(d(P)^2+d(P)e(P))
\]
multiplications of elements in $X$-collected form and $O(d(P)^2e(P)+d(P)e(P)^3)$ bit operations spent for other purposes.
\end{enumerate}
\end{lemmma}

\begin{proof}
For statement (1): First, observe that for each positive integer $m$ and each $a\in\{1,\ldots,m\}$ such that $\gcd(a,m)=1$, the \emph{inverse of $a$ modulo $m$}, i.e., the unique $b\in\{1,\ldots,m\}$ such that $a\cdot b\equiv 1\Mod{m}$, can be computed via the extended Euclidean algorithm using $O((\log{m})^3)$ bit operations, as the algorithm consists of
\begin{itemize}
\item $O(\log{m})$ integer divisions with remainder, of numbers with $O(\log{m})$ binary digits, and each such division has complexity $O(\log{m}\log\log{m})\subseteq O((\log{m})^2)$, see \cite[Introduction]{HH19a}, as well as
\item $O(\log{m})$ backward substitutions, each of which involves $O(1)$ additions, multiplications and reductions modulo $m$ of integers with $O(\log{m})$ binary digits.
\end{itemize}
Now, consider the algorithm {\sc ConstructiveMembershipTest} from \cite[p.~296]{HEO05a}, which can be used to compute the $Y$-exponent vector of any element $h\in G$ given in $X$-collected form, where $Y$ is another pcgs of $G$, whose entries are given in $X$-collected form. We can apply this algorithm with $Y:=X_g$ to obtain the desired output. As for the complexity of this, it is easy to check using Lemma \ref{squareMultiplyLem}(1) and the complexity of the extended Euclidean algorithm that a single call of {\sc ConstructiveMembershipTest} uses $O(d(P)^2+d(P)e(P))$ multiplications of elements in $X$-collected form and $O(d(P)^2e(P)+d(P)e(P)^3)$ other-purpose bit operations, as required.

For statement (2): For each of the $O(d(P)^2)$ defining relations with respect to $X_g$, one first computes the left-hand side of the relation (either a power of a generator to a prime or a conjugate of a generator by another) in $X$-collected form, requiring $O(d(P)+e(P))$ multiplications and $O(d(P)e(P))$ other-purpose bit operations by Lemma \ref{squareMultiplyLem}(1). Then one transforms the result into $X_g$-collected form, requiring $O(d(P)^2+d(P)e(P))$ multiplications and $O(d(P)^2e(P)+d(P)e(P)^3)$ other-purpose bit operations by statement (1), to obtain the right-hand side of the defining relation. Altogether, this process requires $O(d(P)^4+d(P)^3e(P))$ multiplications and $O(d(P)^4e(P)+d(P)^3e(P)^3)$ other-purpose bit operations.

For statement (3): Note that this is tantamount to expressing each $x_i$ in terms of $y_1,\ldots,y_n$ and each $y_i$ in terms $x_1,\ldots,x_n$. The former, for which it is sufficient to express $x_d$ in terms of $y_1,\ldots,y_n$ since $x_i=y_i$ for all $i\not=d$, can be done with $O(d(P)^2+d(P)e(P))$ multiplications and $O(d(P)^2e(P)+d(P)e(P)^3)$ other-purpose bit operations by statement (1), and the latter is clear since we are assuming that $y_d=g$ is given in terms of $x_1,\ldots,x_n$ in the first place.
\end{proof}

Next, we consider the complexity of computing an induced pcgs of a subgroup $H$ from a generating subset of $H$. We follow the approach in \cite[Subsection 8.3.1]{HEO05a}.

\begin{lemmma}\label{inducedPcgsLem}
For any finite solvable group $G$, given through a refined consistent polycyclic presentation $P=\langle X\mid R\rangle$, and any subgroup $U\leq G$, given through a generating tuple $(u_1,\ldots,u_t)$, one can compute an $X$-induced pcgs for $H$ using
\[
O((t+d(P)^2)(d(P)^2+d(P)e(P)))
\]
multiplications of elements in $X$-collected form and
\[
O((t+d(P)^2)(d(P)e(P)^3+d(P)^2e(P)))
\]
bit operations spent outside group element multiplications.
\end{lemmma}

\begin{proof}
Consider the algorithm {\sc InducedPolycyclicSequence} from \cite[p.~294]{HEO05a}, which uses the algorithm {\sc Sift} from \cite[p.~294]{HEO05a} as a subroutine. For {\sc Sift}, it is immediate to check that a single call of it uses $O(d(P)^2+d(P)e(P))$ multiplications and $O(d(P)e(P)^3+d(P)^2e(P))$ other-purpose bit operations by Lemma \ref{squareMultiplyLem}(1) and the complexity of the extended Euclidean algorithm. Consequently, a single iteration of the unique while-loop in {\sc InducedPolycyclicSequence} uses $O(d(P)^2+d(P)e(P))$ multiplications and $O(d(P)e(P)^3+d(P)^2e(P))$ other-purpose bit operations. But the total number of iterations of that while-loop can be bounded as follows: Whenever the if-clause in line 7 is satisfied (which causes $\G$ to be $O(d(P))$ elements larger at the end of the iteration step than at the beginning of the step), an entry $1$ in $Z$ is replaced by a nontrivial element of $G$, which can only happen $O(d(P))$ times, whence the total number of iterations of the while-loop is in $O(t+d(P)^2)$.
\end{proof}

We now consider a modified version of the algorithm ModifyPcgs from \cite[p.~1451]{CEL04a}, given as Algorithm \ref{algo3} below.

\begin{algorithm}
\SetKwInOut{Input}{input}\SetKwInOut{Output}{output}

\Input{A finite solvable group $G$, given through a refined consistent polycyclic presentation $\langle Y\mid S\rangle$, where $Y=(y_1,\ldots,y_n)$; a sequence $\vec{w}=(w_1,\ldots,w_n)$ of numbers that are admissible weights (in the sense of \cite[Subsubsection 3.1.1, p.~1450]{CEL04a}) for the $y_i$ with respect to some fixed normal series $G=N_1\rhd\cdots\rhd N_r\rhd N_{r+1}=\{1\}$ of $G$; an element $g\in G$ in $Y$-collected form
\[
g=y_{d(1)}^{e_{d(1)}}y_{d(2)}^{e_{d(2)}}\cdots y_{d(s)}^{e_{d(s)}},
\]
with $s\in\IN$, $1\leq d(1)<d(2)<\cdots<d(s)\leq n$ and $e_{d(j)}\in\{1,2,\ldots,p_{d(j)}-1\}$ for $j=1,\ldots,s$, where $p_{d(j)}$ is the relative order of $y_{d(j)}$; an admissible weight $u$ for $g$ with respect to the above normal series; a pc group isomorphism $\iota$ (in the sense of Subsection \ref{subsec1P3}) from some other refined consistent polycyclic presentation $P=\langle X\mid R\rangle$ of $G$ to $\langle Y\mid S\rangle$.}
\Output{A refined consistent polycyclic presentation $\langle Z\mid T\rangle$ associated with a certain other pcgs $Z=(z_1,\ldots,z_n)$ of $G$ and a modified weight sequence $(w_1',\ldots,w_n')$ such that $w_i'$ is an admissible weight for $z_i$; moreover, a pc group isomorphism from $\langle X\mid R\rangle$ to $\langle Z\mid T\rangle$.}

\nl \If{$g=1$}{output the presentation $\langle Y\mid S\rangle$, weight sequence $\vec{w}$ and pc group isomorphism $\iota$ from the input and halt.}

\nl \If{$w_{d(1)}<u$}{compute the refined consistent polycyclic presentation $\langle Y'\mid S'\rangle=\langle y'_1,\ldots,y'_n\mid S'\rangle$ of $G$ associated with the pcgs $Y_g$ as in Lemma \ref{elementaryTransfoLem}, and also compute a pc group isomorphism $\iota':\langle Y\mid S\rangle\rightarrow\langle Y'\mid S'\rangle$. Call {\sc ModifyPcgs2} on $\langle Y'\mid S'\rangle$, $(w_1,\ldots,w_{d(1)-1},u,w_{d(1)+1},\ldots,w_n)$, $y_{d(1)}^{-e_{d(1)}}g=(y'_{d(2)})^{e_{d(2)}}\cdots (y'_{d(s)})^{e_{d(s)}}$, $w_{d(1)}$, $\iota'\circ\iota$.}

\nl \Else{Call {\sc ModifyPcgs2} on $\langle Y\mid S\rangle$, $\vec{w}$, $y_{d(2)}^{e_{d(2)}}\cdots y_{d(s)}^{e_{d(s)}}$, $u$, $\iota$.}
\caption{{\sc ModifyPcgs2}}\label{algo3}
\end{algorithm}

Note that unlike ModifyPcgs from \cite[p.~1451]{CEL04a}, this algorithm does not loop over all prime-power components of $g$, so the pcgs from the output may not be a prime-power pcgs (in the sense of \cite[beginning of Subsection 3.1, p.~1450]{CEL04a}) even if the input pcgs is one. The author considered this simplification when encountering difficulties proving that the original algorithm ModifyPcgs has theoretical worst-case complexity bounded by a polynomial in $\Coll(\langle X\mid R\rangle)$ (due to the \enquote{branching} that occurs by looping over the prime-power components); for our Algorithm \ref{algo3}, we can show the following:

\begin{lemmma}\label{modifyPcgsLem}
Algorithm \ref{algo3} terminates after using
\[
O(d(P)^6+d(P)^5e(P)+d(P)^4e(P)^2)
\]
multiplications of elements in $X$-collected form and
\[
O(d(P)^6e(P)+d(P)^5e(P)^3+d(P)^4e(P)^4)
\]
bit operations spent outside group element multiplications.
\end{lemmma}

\begin{proof}
At first glance, this seems straightforward: After the initial user-induced call of the algorithm, this recursive algorithm calls itself $O(d(P))$ times, and the complexity of the computations between two calls can be handled by Lemma \ref{elementaryTransfoLem}. There is, however, a subtlety to be taken into account: The group presentation to be modified is changed along the way, and the straightforward approach would always apply the fixed general multiplication algorithm $\Acal$ to the currently considered presentation $P'$ of $G$, for which it is not clear whether its worst-case multiplication complexity $\Coll(P')$ can be suitably bounded in terms of $\Coll(P)$. We can circumvent this though, by emulating these other multiplication algorithms over $P=\langle X\mid R\rangle$ as follows: At the beginning of each iteration step, we have an isomorphism $P\rightarrow P'$, and we want to subject $P'=\langle X'\mid R'\rangle$ to another elementary transformation step to obtain a presentation $P''=\langle X''\mid R''\rangle$ and compute an isomorphism $P\rightarrow P''$. By Lemma \ref{elementaryTransfoLem}(2,3), one can compute $P''$ as well as an isomorphism $P'\rightarrow P''$ using
\[
O(d(P')^4+d(P')^3e(P'))=O(d(P)^4+d(P)^3e(P))
\]
multiplications of elements in $X'$-collected form and
\[
O(d(P')^4e(P')+d(P')^3e(P')^3)=O(d(P)^4e(P)+d(P)^3e(P)^3)
\]
bit operations spent outside group element multiplications. We follow that approach, but whenever we would normally perform a multiplication of elements in $X'$-collected form using the algorithm $\Acal$, we instead bring the elements into $X$-collected form (using the inverse of the known isomorphism $P\rightarrow P'$, note our convention on pc group isomorphisms from the end of Subsection \ref{subsec1P3}), perform a multiplication over $P$ and transform the result back into $X'$-collected form. The first step requires us to make $O(d(P))$ substitutions, followed by $O(d(P))$ power computations over $P$ and $O(d(P))$ calls of the multiplication algorithm for $P$, overall accounting for $O(d(P)^2+d(P)e(P))$ multiplications of elements in $X$-collected form and $O(d(P)^2e(P))$ other-purpose bit operations by Lemma \ref{squareMultiplyLem}(1). For the last step, we use the algorithm {\sc ConstructiveMembershipTest} from \cite[p.~296]{HEO05a}, for which it was already observed in the proof of Lemma \ref{elementaryTransfoLem}(1) that a single call takes $O(d(P)^2+d(P)e(P))$ multiplications and $O(d(P)^2e(P)+d(P)e(P)^3)$ other-purpose bit operations. In total, computing $P''$ and an isomorphism $P'\rightarrow P''$ therefore costs us
\begin{align*}
&O((d(P)^2+d(P)e(P))\cdot(d(P)^4+d(P)^3e(P))) \\
&=O(d(P)^6+d(P)^5e(P)+d(P)^4e(P)^2)
\end{align*}
multiplications of elements in $X$-collected form and
\begin{align*}
&O(d(P)^4e(P)+d(P)^3e(P)^3+(d(P)^4+d(P)^3e(P))\cdot(d(P)^2e(P)+d(P)e(P)^3)) \\
&=O(d(P)^6e(P)+d(P)^5e(P)^3+d(P)^4e(P)^4)
\end{align*}
other-purpose bit operations. An isomorphism $P\rightarrow P''$ can be computed through composing the known isomorphisms $P\rightarrow P'$ and $P'\rightarrow P''$. More precisely, one can proceed as follows: First, express the elements of $X''$ in $X$-collected form. By going via $P'$, this involves $O(d(P)^2)$ substitutions, followed by $O(d(P)^2)$ computations of powers of elements in $X$-collected form and $O(d(P)^2)$ multiplications of elements in $X$-collected form. Hence, by Lemma \ref{squareMultiplyLem}(1), these computations require $O(d(P)^3+d(P)^2e(P))$ multiplications of elements in $X$-collected form and $O(d(P)^3e(P))$ other-purpose bit operations. Then, once the elements of $X''$ have been expressed in $X$-collected form, proceed to expressing the elements of $X$ in $X''$-collected form. The approach is analogous, but similarly to before, whenever we would carry out a multiplication of elements in $X''$-collected form, we substitute the now known $X$-collected forms of the elements of $X''$, carry out $O(d(P))$ power computations and multiplications of elements in $X$-collected form and bring the result back into $X''$-collected form via {\sc ConstructiveMembershipTest}. This means that per emulation of a single multiplication of elements in $X''$-collected form, we need $O(d(P)^2+d(P)e(P))$ multiplications of elements in $X$-collected form as well as $O(d(P)^2e(P)+d(P)e(P)^3)$ other-purpose bit operations. Hence, in total, expressing the elements of $X$ in $X''$-collected form takes
\[
O((d(P)^3+d(P)^2e(P))\cdot(d(P)^2+d(P)e(P)))=O(d(P)^5+d(P)^4e(P)+d(P)^3e(P)^2)
\]
multiplications of elements in $X$-collected form and
\begin{align*}
&O(d(P)^3e(P)+(d(P)^3+d(P)^2e(P))\cdot(d(P)^2e(P)+d(P)e(P)^3)) \\
&=O(d(P)^5e(P)+d(P)^4e(P)^3+d(P)^3e(P)^4)
\end{align*}
other-purpose bit operations.
\end{proof}

Note that the proof of Lemma \ref{modifyPcgsLem} does \emph{not} show that if $P$ and $Q$ are two refined consistent polycyclic presentations of the same abstract group, then
\begin{align*}
\Coll(Q)&\in O((d(P)^2+d(P)e(P))\Coll(P)+d(P)^2e(P)+d(P)e(P)^3) \\
&\subseteq O(\ell(P)^2\Coll(P)+\ell(P)^4)
\end{align*}
and vice versa, since for the proof idea to work, it is crucial that we know a pc group isomorphism $P\rightarrow Q$ in the first place.

The price which we pay for our simplification of the algorithm ModifyPcgs from \cite[p.~1451]{CEL04a} is that we cannot proceed completely analogously to there, but this will not be a problem. We note the following analogue of \cite[Lemma 8]{CEL04a}, which will be the basis of our further arguments:

\begin{lemmma}\label{modifyExhibitLem}
Let $G$ be a finite solvable group, given through a refined consistent polycyclic presentation $P=\langle X\mid R\rangle$, let $G=N_1\rhd N_2\rhd\cdots\rhd N_r\rhd N_{r+1}=\{1\}$ be a normal series in $G$, and let $Y_j$, for $j=2,\ldots,r$, be an $X$-induced pcgs of $N_j$ (whose entries are given in $X$-collected form). Moreover, let $P'$ resp.~$\alpha$ be the refined consistent polycyclic presentation of $G$ resp.~the pc group isomorphism $P\rightarrow P'$ obtained by successive applications of Algorithm \ref{algo3}, starting with the input containing the presentation $P$, the admissible weight sequence $(1,\ldots,1)$ for the generators of $P$ and the identity isomorphism $P\rightarrow P$, and successively modifying, for $j=2,\ldots,r$, by each entry of $Y_j$ with regard to the admissible weight $u_j=j$ (note that this includes the computation of a new collected form for each entry of each $Y_j$ once one wants to modify by that entry).

Then $P'$ is the polycyclic presentation associated with a pcgs $Y$ of $G$ such that for $j=1,\ldots,r$, the sequence of entries in $Y$ that are displayed in the final output to have admissible weight at least $j$ form a pcgs for $N_j$; in particular, the displayed admissible weight of each entry of $Y$ is the final weight (in the sense of \cite[Subsubsection 3.1.1, p.~1450]{CEL04a}) of that entry with regard to the normal series, and $Y$ exhibits (in the sense of \cite[beginning of Section 2, p.~1446]{CEL04a}) each member $N_j$ of the normal series. Moreover, the described computational process requires
\[
O(d(P)^8+d(P)^7e(P)+d(P)^6e(P)^2)
\]
multiplications of elements in $X$-collected form and
\[
O(d(P)^8e(P)+d(P)^7e(P)^3+d(P)^6e(P)^4)
\]
bit operations spent outside group element multiplications.
\end{lemmma}

\begin{proof}
All assertions apart from the one on the complexity can be proved similarly to \cite[Lemma 8]{CEL04a}, which requires one to first prove an analogue of \cite[Lemma 7]{CEL04a}. However, the situation here is simpler than in \cite[Lemma 8]{CEL04a}, due to our assumption that each pcgs $Y_j$ is $X$-induced, and this simplified situation allows for an alternative proof, which is less elegant, but yields a complete understanding of the effects of the calls of Algorithm \ref{algo3} with which Lemma \ref{modifyExhibitLem} is concerned; we will now give this proof.

First, let us note that when $Z$ is a pcgs of $G$ and $g\in G$, then the notions of $Z$-depth and $Z_g$-depth (see Lemma \ref{elementaryTransfoLem}(1) for the meaning of the notation $Z_g$) of an element of $G$ coincide. It follows that no matter at which stage of the computational process we are, the currently considered pcgs $Z$ of $G$ (which is either left the same or replaced by a pcgs of the form $Z_g$ in the next call of {\sc ModifyPcgs2}) induces the same notion of depth as the original pcgs $X$, and throughout the proof, we will only speak of the \emph{depth} of an element of $G$ (without referring to a particular pcgs of $G$).

Setting $Y_1:=X$, one can show the following by induction on $j$: For each $j=2,3,\ldots,r$, when one has just finished modifying by the elements of $Y_{j-1}$, resulting in a modified pcgs $Z_{j-1}$ and weight sequence $\vec{w}_{j-1}$, then these two tuples can be characterized as follows: For $d=1,2,\ldots,d(P)$,
\begin{itemize}
\item the $d$-th entry of $\vec{w}_{j-1}$ is the largest element $w(j-1,d)\in\{1,2,\ldots,j-1\}$ such that $N_{w(j-1,d)}$ has an element of depth $d$, and
\item the $d$-th entry of $Z_{j-1}$ is the unique entry of $Y_{w(j-1,d)}$ of depth $d$.
\end{itemize}
Indeed, this is true by definition for $j=2$. For the induction step, assume that it is true for some $j\in\{2,3,\ldots,r-1\}$. Note that the assumption that $Y_j$ is $X$-induced implies that no two distinct entries of $Y_j$ have the same depth. We are done with the inductive proof if we can show that the overall effect of modifying by the entries of $Y_j$ is that each entry of $Z_{j-1}$ which has the same depth as one of the entries of $Y_j$ is replaced by that unique entry of $Y_j$, that the corresponding entries of $\vec{w}_{j-1}$ are replaced by $j$, and that nothing else is changed.

To that end, write $Y_j=(y_{j,1},y_{j,2},\ldots,y_{j,\ell(j)})$, and note that since $Y_j$ is induced, its entries are ordered by increasing depth. We prove by induction on $s=0,1,\ldots,\ell(j)$ that after successively modifying by $y_{j,1},\ldots,y_{j,s}$, each entry of $Z_{j-1}$ which has the same depth as one of $y_{j,1},\ldots,y_{j,s}$ is replaced by that unique element, that the corresponding entries of $\vec{w}_{j-1}$ are replaced by $j$, and that nothing else is changed. The induction base, $s=0$, is vacuously true, so assume that $1\leq s<\ell(j)$. We consider the effect of modifying by $y_{j,s+1}$. Say the current pcgs is $Z=(z_1,\ldots,z_{d(p)})$, the current weight sequence is $\vec{w}=(w_1,\ldots,w_{d(P)})$, and write
\[
y_{j,s+1}=z_{d(1)}^{e_{d(1)}}\cdots z_{d(m)}^{e_{d(m)}}
\]
in $Z$-collected form. Then the following happens when modifying by $y_{j,s+1}$: First, the $d(1)$-th entry of $Z$ is replaced by $y_{j,s+1}$, and the $d(1)$-th entry of $\vec{w}$ is replaced by $j$. In its subsequent self-call, Algorithm \ref{algo3} will check whether the \enquote{tail}
\[
z_{d(2)}^{e_{d(2)}}\cdots z_{d(m)}^{e_{d(m)}}
\]
of $y_{j,s+1}$ is to replace the depth $d(2)$ entry of the pcgs. But the assigned weight used for that tail is $w_{d(1)}\in\{1,2,\ldots,j-1\}$, and since this is an admissible weight for the tail, by the inner and outer induction hypotheses, the entry of the pcgs with depth $d(2)$ has assigned weight at least $w_{d(1)}$, and so that pcgs entry will not be replaced by the tail. An analogous argument shows that also none of the other tails
\[
z_{d(k)}^{e_{d(k)}}\cdots z_{d(m)}^{e_{d(m)}},
\]
with which the subsequent self-calls of Algorithm \ref{algo3} are concerned, will replace the corresponding pcgs entry. Hence, indeed, the only changes happening when modifying by $y_{j,s+1}$ are that the $d(1)$-th entries of $Z$ and $\vec{w}$ are replaced by $y_{j,s+1}$ and $j$ respectively. This concludes the inner and thus also the outer induction.

Now, by the above characterization of the pcgs and weight sequence after successively modifying by the entries of $Y_2,Y_3,\ldots,Y_j$ (which was the subject of the outer induction), applied with $j:=r$, we find that in the final result, for each $w\in\{1,2,\ldots,r+1\}$, the pcgs entries with displayed weight at least $w$ are elements of $N_w$, and their depths are just those positive integers that occur as the depth of some element of $N_w$; consequently, the final pcgs exhibits $N_w$, as required.

As for the complexity assertion, just note that the number of (non-self) calls of Algorithm \ref{algo3} in the described modification process is in $O(d(P)^2)$ and use Lemma \ref{modifyPcgsLem} (and note that the complexity of computing the new collected forms for the entries of the $Y_j$ via {\sc ConstructiveMembershipTest} from \cite[p.~296]{HEO05a} is non-dominating, see the proof of Lemma \ref{elementaryTransfoLem}(1)).
\end{proof}

As a final preparation for the proof of Theorem \ref{mainTheo3}, we note:

\begin{lemmma}\label{elAbSeriesLem}
For any finite solvable group $G$, given through a refined consistent polycyclic presentation $P=\langle X\mid R\rangle$, one can compute the following, using
\[
O(d(P)^6+d(P)^5e(P))
\]
multiplications of elements in $X$-collected form and
\[
O(d(P)^5e(P)^3+d(P)^6e(P))
\]
bit operations spent outside group element multiplications:

\begin{itemize}
\item the length $r$ of the LG-series in $G$,
\item tuples $S_1,\ldots,S_r$ of elements of $G$ in $X$-collected form such that for $i=1,\ldots,r$, one has that $S_i$ is an induced pcgs of the $i$-th term $G_i$ in the LG-series $G=G_1\rhd G_2\rhd\cdots\rhd G_r\rhd G_{r+1}=\{1\}$ of $G$.
\end{itemize}
\end{lemmma}

\begin{proof}
Throughout the proof, we use the notation and terminology from \cite[Subsection 2.1]{CEL04a}. We proceed in the following three steps:

\begin{enumerate}
\item Compute (induced polycyclic generating sequences of the members of) the lower nilpotent series of $G$.
\item Compute the refinement of the lower nilpotent series of $G$ by the lower elementary central series of each factor (the nilpotent-central series of $G$).
\item Compute the elementary abelian nilpotent-central series, i.e., the LG-series of $G$, by further refining the nilpotent-central series of $G$ using the Sylow subgroups of its factors.
\end{enumerate}

For Step (1): Assume that $Y$ is an $X$-induced pcgs of a subgroup $H\leq G$. Then by \cite[Lemma 8.39]{HEO05a}, $[X,Y]:=\{[x,y]\mid x\in X,y\in Y\}$ is a generating subset of $[G,H]$. Computing all $O(d(P)^2)$ members of $[X,Y]$ requires $O(d(P)^3)$ multiplications (of elements in $X$-collected form) and $O(d(P)^3e(P))$ other-purpose bit operations (as the number of multiplications resp.~other-purpose bit operations needed for computing a single commutator in $G$ is in $O(d(P))$ resp.~in $O(d(P)e(P))$ by Lemma \ref{squareMultiplyLem}(1)). Once $[X,Y]$ has been computed, one can gain an induced pcgs for $[G,H]$ from it using $O(d(P)^4+d(P)^3e(P))$ multiplications and $O(d(P)^3e(P)^3+d(P)^4e(P))$ other-purpose bit operations by Lemma \ref{inducedPcgsLem}. This allows us to compute the smallest term in the lower central series of $G$ using $O(d(P)^5+d(P)^4e(P))$ multiplications and $O(d(P)^4e(P)^3+d(P)^5e(P))$ other-purpose bit operations (we know when to stop by comparing the lengths of the computed induced polycyclic generating sequences), and iterating this $O(d(P))$ times, one computes the entire lower nilpotent series of $G$ using $O(d(P)^6+d(P)^5e(P))$ multiplications and $O(d(P)^5e(P)^3+d(P)^6e(P))$ other-purpose bit operations.

For Step (2): Focus on a single factor $G_n/G_{n+1}$ in the lower nilpotent series of $G$ (there are $O(d(P))$ such factors). For an induced pcgs $Y$ of some subgroup $H\leq G$, define $Y^{\sharp}:=\{y^{\rord(y)}\mid y\in Y\}$, where $\rord(y)$ denotes the relative order of $y$ (with respect to either of $X$ or $Y$). Let $Y_1$ resp.~$Z$ be the induced pcgs of $G_n$ resp.~of $G_{n+1}$ computed in Step (1). For $i\geq 2$, we recursively compute an induced pcgs $Y_i$ of the subgroup $H_i$ of $G$ projecting onto $\lambda_i(G_n/G_{n+1})$ (see \cite[Subsection 2.1]{CEL04a} for the meaning of this notation) under the canonical projection $G\rightarrow G/G_{n+1}$. We do so as in the proof of Lemma \ref{inducedPcgsLem}, by applying {\sc InducedPolycyclicSequence} from \cite[p.~294]{HEO05a} to the generating subset $[Y_1,Y_{i-1}]\cup Y_{i-1}^{\sharp}\cup Z$ of $H_i$. The computation of $[Y_1,Y_{i-1}]\cup Y_{i-1}^{\sharp}\cup Z$ requires $O(d(P)^3+d(P)e(P))$ multiplications as well as $O(d(P)^3e(P))$ other-purpose bit operations by Lemma \ref{squareMultiplyLem}(1), and by Lemma \ref{inducedPcgsLem}, the subsequent application of {\sc InducedPolycyclicSequence} takes $O(d(P)^4+d(P)^3e(P))$ multiplications and $O(d(P)^3e(P)^3+d(P)^4e(P))$ other-purpose bit operations. Therefore, a single factor in the lower nilpotent series of $G$ can be refined using $O(d(P)^5+d(P)^4e(P))$ multiplications and $O(d(P)^4e(P)^3+d(P)^5e(P))$ other-purpose bit operations, and it takes $O(d(P)^6+d(P)^5e(P))$ multiplications and $O(d(P)^5e(P)^3+d(P)^6e(P))$ other-purpose bit operations to refine the entire series.

For Step (3): Focus on a single factor $G_{i,j}/G_{i,j+1}$ in the nilpotent-central series of $G$. Let $Y_0$ resp.~$Z$ be the induced pcgs of $G_{i,j}$ resp.~of $G_{i,j+1}$ known from Step (2). Moreover, let $p_1,\ldots,p_{\ell}$ be the prime divisors of $|G_{i,j}/G_{i,j+1}|$ (i.e., the relative orders of the entries of $Y_0$ whose depth is not among the depths of the entries of $Z$). Set $H_0:=G_{i,j}$, and recursively compute an induced pcgs $Y_{k+1}$ of $H_{k+1}:=H_k^{p_k}G_{i,j+1}$ by applying {\sc InducedPolycyclicSequence} from \cite[p.~294]{HEO05a} to the generating subset $Y_k^{p_k}\cup Z=\{y^{p_k}\mid y\in Y_k\}\cup Z$ of $H_{k+1}$. By Lemma \ref{squareMultiplyLem}(1), it takes $O(d(P)^2+d(P)e(P))$ multiplications as well as $O(d(P)^2e(P))$ other-purpose bit operations to compute $Y_k^{p_k}\cup Z$, and it takes another $O(d(P)^4+d(P)^3e(P))$ multiplications and $O(d(P)^3e(P)^3+d(P)^4e(P))$ other-purpose bit operations to apply {\sc InducedPolycyclicSequence} to it (see Lemma \ref{inducedPcgsLem} and its proof). Since $\ell\in O(d(P))$, it therefore takes $O(d(P)^5+d(P)^4e(P))$ multiplications and $O(d(P)^4e(P)^3+d(P)^5e(P))$ other-purpose bit operations to refine a single factor $G_{i,j}/G_{i,j+1}$, and thus Step (3) in total takes $O(d(P)^6+d(P)^5e(P))$ multiplications and $O(d(P)^5e(P)^3+d(P)^6e(P))$ other-purpose bit operations.
\end{proof}

\begin{proof}[Proof of Theorem \ref{mainTheo3}]
By Lemmas \ref{modifyExhibitLem} and \ref{elAbSeriesLem}, we can compute, using
\[
O(d(P)^8+d(P)^7e(P)+d(P)^6e(P)^2)
\]
multiplications of elements in $X$-collected form and
\[
O(d(P)^8e(P)+d(P)^7e(P)^3+d(P)^6e(P)^4)
\]
other-purpose bit operations, an isomorphism from $P$ to a refined consistent polycyclic presentation $P'$ of $G$ associated with a pcgs of $G$ that exhibits (in the sense of \cite[beginning of Section 2, p.~1446]{CEL04a}) each member of the LG-series of $G$, as well as the associated sequence of final weights (in the sense of \cite[Subsubsection 3.1.1, p.~1450]{CEL04a}). By \cite[Lemma 5, p.~1450]{CEL04a}, to achieve the same situation with regard to a presentation $\tilde{P}$ with an associated pcgs that even \emph{refines} (in the sense of \cite[beginning of Section 2, p.~1446]{CEL04a}) that series, we just have to order the pcgs entries by increasing weight (preserving the order among entries with the same weight) and accordingly relabel variable indices in the defining relations and in the images of the isomorphism as well as reorder the images of the inverse isomorphism and the weight sequence entries, which can all be done in $O(d(P)^8)$ bit operations. The total number of bit operations needed therefore lies in
\begin{align*}
O(&(d(P)^8+d(P)^7e(P)+d(P)^6e(P)^2)\Coll(P) \\
&+d(P)^8e(P)+d(P)^7e(P)^3+d(P)^6e(P)^4) \\
&\subseteq O(\ell(P)^8\Coll(P)+\ell(P)^{10}),
\end{align*}
as required.
\end{proof}

\subsection{Complexity analysis for the rest of Algorithms \ref{algo1} and \ref{algo2}}\label{subsec3P2}

Note that we have already given details on how to carry out the computations in the remaining steps of the two algorithms in Subsection \ref{subsec2P1}. Set $P:=\langle X\mid R\rangle$, the presentation of $G$ from the input of Algorithms \ref{algo1} and \ref{algo2}. We do not go into as much detail as in Subsection \ref{subsec3P1} here, particularly since at this point, we are only interested in proving subexponential complexity. We will, however, mention what we consider to be the most important ideas of the remaining complexity analysis.

For Algorithm \ref{algo1}, observe that Theorem \ref{mainTheo3} does not cover the last bullet point (the computation of the $\vec{g}$-exponent matrix of $\alpha$), but this can be easily handled via the computed isomorphism $P\rightarrow\langle Y\mid S\rangle$. The coefficients of the matrices $M_i$ from Step 6 can be read off directly from the $\vec{g}$-exponent matrix of $\alpha$. Moreover, the number $r$ and the dimension of each $M_i$ are bounded from above by the composition length $d(P)$ of $G$, hence also by the presentation input length $\ell(P)$. That the computation of the numbers $o_i$ requires expectedly subexponentially (in $\ell(P)$) many bit operations therefore follows from \cite{CL97a} and the considerations on integer factorization from Subsection \ref{subsec2P1}. After computing each $o_i$ and their least common multiple $o$ (see Step 8), we compute the $\vec{g}$-exponent matrix $\beta=\alpha^o$ using the square-and-multiply approach from Lemma \ref{squareMultiplyLem}(3) with respect to the collection algorithm of $\langle Y\mid S\rangle$, which is okay since the (subexponential) bound on the multiplication complexity given in \cite[Theorem, p.~2]{Hof04a} only depends on parameters of the abstract group $G$ (one could, of course, also use the \enquote{emulation strategy} from the proof of Lemma \ref{modifyPcgsLem} to reduce the complexity of at least that step to $O(\ell(P)^c\Coll(P))$ bit operations for some absolute constant $c$). The same applies to the potential further computations of automorphism powers in Step 13.

For Algorithm \ref{algo2}, in view of the explanations for Algorithm \ref{algo1} above, we only need to analyze Steps 9 and 11 further. For Step 11, just use a square-and-multiply approach via the formula $\A_{t_1,\alpha_1}\circ\A_{t_2,\alpha_2}=\A_{t_1\alpha_1(t_2),\alpha_1\circ\alpha_2}$. For Step 9: As noted in Subsection \ref{subsec2P1}, we can transform each matrix $M_i$ into its rational canonical form within $O({\ell_i}^c)\subseteq O(\ell(P)^c)$ bit operations (where, by \cite[pp.~4 and 140]{Sto00a}, $c$ may be chosen as $2.4$) and directly read off the invariant factors $P_i$ from it. Before trying to factorize the polynomial $P_i$ with Berlekamp's Las Vegas algorithm from \cite{Ber70a}, we first test whether it is irreducible using Rabin's algorithm, see \cite[Lemma 1]{Rab80a}; note that this also requires us to first determine the prime factors of $\deg{P_i}\leq d(P)\in O(\ell(P))$. As long as there is a factor in the intermediate factorizations of $P_i$ which is found to be reducible by Rabin's algorithm, we use Berlekamp's algorithm to split that factor up further in expectedly polynomially many bit operations. This combination of irreducibility testing and, where applicable, searching for smaller factors only needs to be applied $O(\deg{P_i})\subseteq O(\ell(P))$ times until the polynomial $P_i$ has been fully factored. The rest of the complexity analysis for Step 9 (i.e., of the computations described in bullet point 5 in Subsection \ref{subsec2P1}) is straightforward.

\section{Concluding remarks}\label{sec4}

We conclude with some remarks on computational problems related to the ones discussed in this paper. In Subsection \ref{subsec4P1}, we present a problem that is probably computationally hard, and in Subsection \ref{subsec4P2}, we talk about other problems for which subexponential-time algorithms can be given.

\subsection{Contrast to the cycle membership problem}\label{subsec4P1}

In the context of this paper, the following problem, which we call the \enquote{cycle membership problem}, may also be of interest: Given a finite solvable group $G$, an automorphism $\alpha$ of $G$ and elements $g_1,g_2\in G$, decide if $g_2$ lies on the cycle of $g_1$ under $\alpha$, i.e., whether there exists $n\in\IZ$ such that $\alpha^n(g_1)=g_2$. We now discuss a connection between this problem and the discrete logarithm problem, which indicates that the cycle membership problem is probably hard in general, even for the special case $G=\IZ/p\IZ$ with $p$ a prime.

Assume that we have an efficient algorithm (say, requiring $o(|G|)$ bit operations as $|G|\to\infty$, which is asymptotically better than the obvious brute-force approach) for solving the cycle membership problem. Then in particular, we have an efficient (requiring $o(p)$ bit operations as $p\to\infty$) algorithm to decide for a given triple $(p,a,b)$, where $p$ is a prime and $a,b\in\{1,\ldots,p-1\}$, whether $b$ is a power of $a$ modulo $p$. We claim that we then also have an efficient algorithm for the following promise problem, which is a restricted version of the discrete logarithm problem: For a given triple $(p,a,b)$ as above, but where additionally, the multiplicative order of $a$ modulo $p$ is a power of $2$, decide whether $b$ is a power of $a$ modulo $b$, and if so, output the unique $e\in\{0,\ldots,\ord_p(a)-1\}$ such that $b\equiv a^e\Mod{p}$. Indeed, writing the multiplicative order of $a$ modulo $p$ as $2^o$ with $o\in\IN$, if $b$ is a power of $a$ modulo $p$, then all $f\in\IN$ such that $b\equiv a^f\Mod{p}$ are congruent modulo $2^o$, and hence their first $o$ binary digits (starting to count from the ones digit) coincide. So if we assume, aiming for a recursive approach, that we know already the first $i$ such digits $c_0,\ldots,c_{i-1}$ for some $i\in\{0,\ldots,o-1\}$, then we can find the next digit by deciding whether $b\cdot d^{\sum_{j=0}^{i-1}{c_j2^j}}$, with $d$ the multiplicative inverse of $a$ modulo $p$, is a power of $a^{2^{i+1}}$ modulo $p$. We can stop this loop (knowing that the number of digits we have found is precisely $o$ and that they therefore comprise the significant digits of $e$) as soon as $a^{2^i}\equiv 1\Mod{p}$.

\subsection{Computational problems in the context of finite dynamical systems}\label{subsec4P2}

A \emph{finite dynamical system} (\emph{FDS}) is a finite set $S$ together with a function $f:S\rightarrow S$. People working on FDSs $(S,f)$ are usually interested in the behavior of $f$ under iteration; an important special case with many applications is when $S=k^n$ is a Cartesian power of a finite field $k$ and $f:k^n\rightarrow k^n$ is written as a polynomial function (see, for instance, \cite{CJLS06a}, which includes several references concerning applications in natural sciences such as biology, and \cite{OS10a}, which discusses potential cryptographic applications). Several computational problems in the context of FDSs are also interesting to study, such as the following, given an FDS $(S,f)$ and an element $s\in S$:

\begin{itemize}
\item Compute the size of the \emph{orbit of $s$ under $f$}, i.e., compute $|\{f^n(s)\mid n\in\IN\}|$.
\item Decide whether $s$ is periodic under $f$, i.e., whether $f^n(s)=s$ for some $n\in\IN^+$.
\item Compute the \emph{preperiod length of $s$ under $f$}, i.e., compute the smallest $t\in\IN$ such that $f^t(s)$ is periodic.
\end{itemize}

Our Algorithm \ref{algo2} solves the first problem in the special case where $S$ is a (finite) solvable group $G$ and $f$ is a bijective affine map of $G$. Without giving a detailed analysis, we note that the other two problems for $(S,f)=(G,\varphi)$, a finite solvable group together with an endomorphism, admit deterministic solution algorithms with complexity in $O(\ell(P)^c\Coll(P))$ for some positive constant $c$, where $P$ is the refined consistent polycyclic presentation through which $G$ is given. This is because by \cite[Theorem 4.2]{Car13a} and Lagrange's theorem, the preperiod length of any $g\in G$ under $\varphi$ is at most $\lfloor\log_2{|G|}\rfloor$, and in particular, the subgroup of $G$ consisting of the periodic points of $\varphi$ is just the image of $\varphi^{\lfloor\log_2{|G|}\rfloor}$.

\section{Acknowledgements}\label{sec6}

The author would like to thank Alexander Hulpke for providing his GAP implementation of the algorithm from \cite[Section 12.2, pp.~481f.]{DF04a}, which is used in the author's GAP implementations of Algorithms \ref{algo1} and \ref{algo2}, and also for providing helpful answers to questions on the functionality of order and cycle length computations in GAP raised by the author in the GAP Forum.

\end{document}